\documentclass[a4paper,egregdoesnotlikesansseriftitles,final]{scrartcl}
\usepackage[utf8]{inputenc}
\usepackage[british]{babel}
\usepackage[T1]{fontenc}
\usepackage{lmodern}
\usepackage{microtype}
\usepackage{amsfonts,amsthm,amssymb,amsmath}
\usepackage{mathtools}
\usepackage{xcolor}
\usepackage{graphicx}
\usepackage{tikz}
\usetikzlibrary{arrows}
\usepackage[colorlinks,unicode,bookmarks,pdfusetitle,linkcolor={red!60!black},citecolor={green!60!black},urlcolor={blue!60!black},bookmarksnumbered,bookmarksopen]{hyperref}
\usepackage[capitalize,nameinlink]{cleveref}
\crefname{enumi}{}{}
\crefname{property}{property}{properties}
\crefname{LEM}{Lemma}{the Lemmas}
\usepackage{enumitem}
\usepackage{thmtools}
\usepackage{cite}
\usepackage[abbrev]{amsrefs}
\usepackage{doi}
\renewcommand{\PrintDOI}[1]{\doi{#1}}

\newtheorem{THM}{Theorem}[section]
\newtheorem{LEM}[THM]{Lemma}
\newtheorem{COR}[THM]{Corollary}
\newtheorem{QUESTION}[THM]{Question}
\newtheorem{PROP}[THM]{Proposition}
\newtheorem{COUNTEREX}[LEM]{Counterexample}

\newtheorem{OBS}[THM]{Observation}
\theoremstyle{definition}
\newtheorem{EX}[THM]{Example}

\newcommand{\abs}[1]{\lvert#1\rvert}

\renewcommand{\phi}{\varphi}

\newcommand{\N}{\mathbb{N}}

\newcommand{\sm}{\smallsetminus}
\newcommand{\cA}{\mathcal{A}}

\newcommand{\cN}{\mathcal{N}}

\newcommand{\cV}{\mathcal{V}}

\newcommand{\cupdot}{\dot\cup}

\title{Agile Sets in Graphs}
\author{Christian Elbracht \and Jakob Kneip \and Maximilian Teegen}
\date{10th September, 2021}

\begin{document}
\maketitle
\begin{abstract}
A set of vertices in a graph is \emph{agile} if, however we partition the set into two parts, we can always find two vertex-disjoint connected subgraphs where one covers the first and the other the second part.
We present a characterization for the existence of large agile sets in terms of $K_{2,k}$ and large strip minors.
\end{abstract}
\section{Introduction}\label{sec:agile}
Seymour \cite{SEYMOUR1980293} and, independently, Thomassen \cite{THOMASSEN1980371} considered, given an ordered set 
$Z=(x_1,x_2,\dots,x_k,y_1,y_2,\dots,y_k)$ of vertices in a graph $G$, the 
question whether there exists a $Z$-linkage in $G$. Here, a 
\emph{$Z$-linkage} 
consists of vertex-disjoint paths $P_1,P_2,\dots P_k$ between $x_i$ and $y_i$. 
Both Seymour and Thomassen independently characterized when a graph contains 
a $Z$-linkage for a given ordered set $Z$ of size $4$.

Instead of considering this question for larger $k$ let us generalize the 
special case of $k=2$ in a different direction. Let us say that a pair 
$(X_1,X_2)$ of disjoint vertex sets in $G$ is \emph{independent} if we 
find two disjoint trees $T_1,T_2$ in $G$ so that $T_1$ contains all of 
$X_1$ and $T_2$ contains all of $X_2$.

Weißauer \cites{DWpersonal,Mathoverflow_independent} invented this notion to give a possible definition of when a vertex 
set is, in some sense, dense in the graph: let us say that a vertex set $X$ is 
\emph{agile} 
if for every partition $X=X_1\cupdot X_2$ the pair $(X_1,X_2)$ is independent.

A typical example of a large agile set can be found in the complete bipartite graph 
$K_{2,k}$: the set $X$ of all the degree-$2$ vertices in this graph is agile, 
since for any partition $X_1\cupdot X_2$ of this set we can obtain disjoint 
trees 
$T_1$ and $T_2$ by adding one of the degree-$k$ vertices to $X_1$ and the other 
to $X_2$.

Moreover, this notion of agile sets is well-behaved under the minor relation, 
since if $H$ is a minor of $G$ containing an agile set $X$, then picking an 
arbitrary vertex from every branch set corresponding to a vertex in $X$ results 
in an agile set contained in $G$.

In light of these observations, Weißauer asked \cite{DWpersonal,Mathoverflow_independent} whether, at 
least qualitatively, a graph contains a large agile set if and only if the 
graph 
contains a large $K_{2,k}$ as a minor. More precisely, Weißauer asked the 
following:
\begin{QUESTION}\label{question:agile_k2n}
 Is it true that for every $k$ there exists an $m$ such that every graph with 
an 
agile set of size at least $m$ contains $K_{2,k}$ as a minor?
\end{QUESTION}
%   Let $G =(V,E)$ be a graph.
%   A set $A \subseteq V$ is \emph{agile} in $G$ if for every partition $A = 
% X\cupdot Y$ there
%   are vertex-disjoint connected $T_X, T_Y \subseteq G$ with $X \subseteq 
% T_X$, $Y \subseteq T_Y$.
%   
% \begin{QUESTION}
%   For every $k \in \N$ there is some $f(k)$ such that every graph with an 
% agile set of size $f(k)$ has a $K_{2,k}$--minor.
% \end{QUESTION}

For $k=2$ this is the case, since all graphs without a $K_{2,2}$-minor are 
outerplanar and thus cannot contain an agile set of size $4$. This was 
already observed by Weißauer:
\begin{OBS}[\cite{DWpersonal}]\label{agile:k22}
If $G$ does not contain a $K_{2,2}$-minor, then $G$ cannot contain an agile set 
of size $\ge 4$, as in that case $G$ is outerplanar.
\end{OBS}
In this paper we will analyse graphs with larger agile sets and answer 
\cref{question:agile_k2n}. We will find out that, while the answer to 
\cref{question:agile_k2n} is `yes' for $k\le 4$, for larger $k$ the question must 
be answered negatively. However, this is only the case due to one special 
additional type of graph, and consequently we will be able to show that there is 
a 
function $f\colon\N\to\N$ such that every graph containing an agile set of size 
$f(k)$ will need to contain a $K_{2,k}$, or this other special type of graph, 
which is called a regular strip of length $k$, as a minor. I.e.\ we show the following:
\begin{restatable}{THM}{agilecharacter}\label{thm:agile_character}
  There exists a function $f\colon \N \to \N$, such that every graph with an agile 
set of size $f(k)$ either contains $K_{2,k}$ or a regular strip of length $k$ 
as a minor.
\end{restatable}

Moreover, we will also consider a possible generalization of the notion of agile sets which we call $k$-agile sets and will qualitatively characterize the existence of these sets in graphs.

This paper is structured as follows: After giving the formal definitions required for this paper in \cref{sec:prelims}, we start in \cref{sec:small} with analysing \cref{question:agile_k2n} for small $k$. We show that the answer to this question is `yes' for $k\le 4$, but `no' for $k>4$, due to one specific construction of a counterexample. In \cref{sec:characterizing} we then show \cref{thm:agile_character}, proving that this type of counterexample is essentially the only one.
In the last section of this paper, \cref{sec:generalizations}, we look at a possible generalization of the notion of agile sets.

\section{Preliminaries}\label{sec:prelims}
\subsection{Graph-theoretic notation}
We follow the graph-theoretic notions from the textbook of Diestel \cite{DiestelBook16noEE}.
In what follows, we briefly recap some important definitions which we need.

An (oriented) \emph{separation of a graph $G$} is an ordered pair $(A,B)$ of subsets of the vertex
set $V(G)$ with the property that $A\cup B=V$ and there is no edge in $G$ from
$A\sm B$ to $B\sm A$. The \emph{order} of a separation is $\abs{A\cap B}$.

Separations come with a natural partial order: $(A, B) \le (C,D)$ if and only if $A \subseteq C$ and $B \supseteq D$.
The \emph{involution} is the function which maps every separation $(A,B)$ to $(B,A)$.

Two separations $(A,B)$ and $(C,D)$ are \emph{nested} if $ (A, B) \le (C, D)$ or $(A, B) \le (D, C)$ or $(B, A) \le (C,D) $, or $(B, A) \le (D, C)$.

A \emph{tree-decomposition} of a graph $G$ is a pair $(T, \cV)$ of a tree $T$ together with a family $\cV = (V_t \mid t \in T)$ of subsets of $V(G)$ called \emph{parts} such that
\begin{enumerate}[label=(T\arabic*)]
    \item $\bigcup_{t\in T} V_t = V(G)$;
    \item for every edge $xy \in E(G)$ there exists some $t \in T$ with $\{x, y\} \subseteq V_t$;
    \item for all $t, t'' \in T$ and all $t'$ on the unique $t$--$t''$-path in $T$ we have 
    $ V_t \cap V_{t''} \subseteq V_{t'} $.
\end{enumerate}

Given an orientation of an edge $tt'$ in the tree of such a tree-decomposition, let $T_t$ and $T_{t'}$ be the components of $T- tt'$ containing $t$ and $t'$, respectively. Then we call \[
    \big(\bigcup_{s\in T_t} V_s,\, \bigcup_{s\in T_{t'}} V_s\big)
\] the separation of $G$ that is \emph{induced} by $tt'$.
Given a tree-decomposition of $G$, the set of all induced separations forms a set of pairwise nested separations of $G$.

Conversely, every set $\cN$ of pairwise nested separations which is closed under involution is the set of induced separations of some tree-decomposition. This is a consequence of Theorem~4.8 from~\cite{confing}:

\begin{LEM}[cf.\ \cite{confing}*{Theorem 4.8}]
    Let $\cN$ be a set of pairwise nested separations of some finite graph $G$ which is closed under involution, then there exists
    a tree-decomposition $(T, \cV)$ of $G$ such that $\cN$ is the set of induced separations of $(T,\cV)$.
    % and the parts of $(T,\cV)$ are precisely the maximal subsets of $V(G)$ which are not separated by ...
\end{LEM}

A \emph{torso} of a tree-decomposition $(T, \cV)$ of $G$ is obtained, for some $t\in T$, from the induced subgraph $G[V_t]$ by making each of the sets $V_t \cap V_{t'}$ complete for every $t' \in T - t$.

% Minor?
\subsection{Basic Lemmas about agile sets}
We can now make some basic observations about the connectivity structure of
graphs which contain agile set.
We start with a lemma, which will allow us to assume that the graphs we consider are $2$-connected:
\begin{LEM}\label{2-connected}
If $G$ is a graph containing an agile set $X$ of size $l\ge 4$, then either $G$ 
is $2$-connected or $G$ contains a proper subgraph still containing an agile 
set 
of size $l$.
\end{LEM}
\begin{proof}
Suppose $G$ is not $2$-connected and let $x$ be a vertex such that the set 
$\{C_1,\dots,C_k\}$ of components of $G-x$ has size at least $2$.

Then there is one component, say $C_1$, such that all but at most one
vertex from $X$ lies in $C_1$: otherwise pick four vertices $v_1,v_2,w_1,w_2$ 
so that neither $v_1$ and $v_2$ nor $w_1$ and $w_2$ lie in the same 
component 
of $G-x$. Then there are no two disjoint trees $T_1,T_2$ in $G$ 
such that $T_1$ contains $v_1$ and $v_2$ and $T_2$ contains $w_1$ and $w_2$, 
as both 
of these trees would need to contain $x$, since every path from $v_1$ to $v_2$ 
uses $x$ and also every path from $w_1$ to $w_2$ uses $x$. 

Moreover, if $v$ is the only vertex in $X\setminus C_1$, the set $X-v+x$ is 
again an agile set. It is easy to see that this set is also an agile set of the 
proper subgraph $C_1\cup \{x\}$ of $G$, thus this subgraph contains an agile 
set 
of size $l$.
\end{proof}
While this lemma will later allow us to essentially assume that our graph 
containing an agile set is $2$-connected, we can, perhaps surprisingly, show 
something similar 
for larger connectivity. Of course, not every graph containing a large agile 
set 
is $3$-connected, as for example $K_{2,k}$ is not, but an agile set still 
behaves nicely with respect to separations of order $2$ or larger. Namely, we 
observe the following:
\begin{OBS}\label{fact:agile_seps}
 Let $G$ be a graph, $(C,D)$ a separation of $G$ and $X$ an agile set contained 
in $G$. 
Then $X\cap C$ and $X\cap D$ are agile in the corresponding torsos, i.e.\ in 
the 
graphs obtained from $G[C]$ and $G[D]$ by making $C\cap D$ complete. 
\end{OBS}
\begin{proof}
Let us show that $X\cap C$ is agile in the torso $H$ corresponding to $C$.
 Given any partition $X_1\cupdot X_2$ of $X\cap C$, we know that, since $X$ is 
agile, 
the pair $(X_1\cup (X\sm C),X_2)$ is independent. Thus, we find disjoint trees 
$T_1,T_2$ in $G$ so that $T_1$ contains $X_1\cup (X\sm C)$ and $T_2$ contains 
$A_2$. Now clearly $T_1\cap C$ and $T_2\cap C$ induce disjoint connected 
subgraphs of $H$ which contain $X_1$ and $X_2$ respectively. Thus, we find the 
desired trees inside these subgraphs.
\end{proof}

\section{\texorpdfstring{\Cref{question:agile_k2n}}{\autoref*{question:agile_k2n}} for small \texorpdfstring{$k$}{k}}\label{sec:small}
We will now give answers to \cref{question:agile_k2n} for small values of $k$, observing that 
the smallest value were \cref{question:agile_k2n} must be answered negatively is $k=5$.
The case $k = 2$ is covered by \cref{agile:k22}.
We now consider $k = 3$, where \cref{2-connected} is all that is needed to answer \cref{question:agile_k2n} positively:
\begin{PROP}
If $G$ does not contain a $K_{2,3}$-minor, then $G$ cannot contain an agile 
set 
of size $5$.
\end{PROP}
\begin{proof}
We may assume by \cref{2-connected} that $G$ is $2$-connected.
 
Since a graph is outerplanar if and only if the graph contains neither 
$K_{2,3}$ nor $K^4$ as a minor, we may suppose that $G$ is either 
outerplanar or contains $K^4$ as a minor. Thus, as by \cref{agile:k22} no 
outerplanar graph contains an agile set of size $5$, we may suppose that $G$ 
contains $K^4$ as a minor. Thus, (since $\Delta(K^4)=3$,) $G$ contains $K_4$ 
as a topological minor. Since every $TK^4$ not equal to $K^4$ contains 
$K_{2,3}$ as a minor, we therefore may suppose that $K^4\subseteq G$.

If there is any other vertex $v\in G$, there are two disjoint paths from $v$ to 
this $K^4$, since $G$ is $2$-connected. However, the $K^4$ together with $v$ 
and 
these $2$ paths again include $K_{2,3}$ as a minor, thus there cannot be such a 
vertex. 

Thus, $G$ would need to be equal to $K^4$, which does not contain an 
agile set of size $5$. Thus, every graph containing an 
agile set of size $5$ must contain $K_{2,3}$ as a minor.
\end{proof}
For $k = 4$ also, a positive answer to \cref{question:agile_k2n} can be given:
\begin{PROP}\label{prop:agileK24}
There exists an $m$ such that, if $G$ does not contain a $K_{2,4}$-minor, then 
$G$ cannot contain an agile set of size $m$.
\end{PROP}
While \cref{prop:agileK24} can be proven directly using either the 
characterization of graphs without $K_{2,4}$-minor obtained by Ellingham, 
Marshall, Ozeki, and Tsuchiya \cite{K_24-Minor-free}, 
or a result by Dieng \cite{K_24} which states that every graph without a 
$K_{2,4}$-minor is obtained from an outerplanar graph by the addition of at 
most 
$2$ vertices, both of these proofs would consist of a rather extensive case 
distinction. Thus, we will not prove \cref{prop:agileK24} directly, instead 
it will turn up as corollary of \cref{thm:agile_character}.

For $k=5$ however, \cref{question:agile_k2n} must be answered negatively, 
as 
shown by the following counterexample:

\begin{COUNTEREX}\label{cex:noK25}
The set of red vertices in the following graph $G$ is agile, however $G$ does 
not contain a $K_{2,5}$-minor.
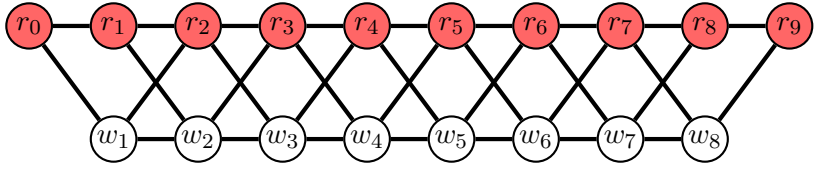
\begin{center}
\begin{tikzpicture}
\tikzstyle{hvertex}=[thick,circle,inner sep=0.cm, minimum size=6mm, 
fill=white, draw=black]
\tikzstyle{hedge}=[ultra thick]
	\def\npoints{9}
	\pgfmathtruncatemacro{\nminusone}{\npoints - 1}
	\pgfmathtruncatemacro{\nminustwo}{\npoints - 2}
	\foreach \i in {0,1,..., \npoints}{
		\node[hvertex,fill=red!60!white] (w\i) at (\i * 
10cm/\npoints,4cm){$r_{\i}$};
	}
	\foreach \i in {1,2,..., \nminusone}{
		\node[hvertex] (v\i) at (\i * 10cm/\npoints,2.5cm){$w_{\i}$};
	}
	\foreach \i in {1,..., \nminustwo}{
		\pgfmathtruncatemacro{\iplus}{\i + 1}
		\draw[hedge] (v\i) -- (v\iplus);
		\draw[hedge] (v\i) -- (w\iplus);
		\draw[hedge] (w\i) -- (v\iplus);
		\draw[hedge] (w\i) -- (w\iplus);
	}
	\draw[hedge] (w0) -- (w1);
	\draw[hedge] (w0) -- (v1);
	\draw[hedge] (w\nminusone) -- (w\npoints);
	\draw[hedge] (v\nminusone) -- (w\npoints);
\end{tikzpicture}
\end{center}
Formally this graph $G$ is constructed as follows: given some $n\in \N$, the 
vertex set of $G$ 
consists of the red vertices $r_0,r_1,\dots r_n$ and the white vertices 
$w_1,\dots w_{n-1}$. The edge set of $G$ consist of an edge between $r_i$ and 
$r_{i+1}$ for all $0\le i<n$, an edge between $w_i$ and $w_{i+1}$ for all $1\le 
i<n-1$ as well as an edge between any $w_i$ and $r_{i+1}$ and any $w_i$ and 
$r_{i-1}$ for all $1\le i<n-1$.
\end{COUNTEREX}
\begin{proof}
It is easy to see that the set $X=\{r_0,\dots,r_n\}$ of the red vertices in $G$ 
is agile:
given a partition $X = X_1 \dot\cup X_2$ we let $T_j = G[X_j \cup \{ w_i \mid r_i \notin X_j \}]$
for $j = 1,2$. Then $T_1$, $T_2$ are disjoint trees containing $X_1$ and $X_2$, respectively.

To see that $G$ does not contain a $K_{2,5}$-minor, suppose for a contradiction that $G$ does contain a $K_{2,5}$-minor. Then we can also find such a minor 
so that the 
branch set of every vertex of degree $2$ in $K_{2,5}$ consists of only a single 
vertex of $G$. Let us denote the branch sets of the vertices of degree $5$ of 
such a $K_{2,5}$-minor in $G$ by $H_1$ and $H_2$.

Now consider the set $I$ of those $i\in \{0,\dots,n\}$ for which $r_i$ or 
$w_i$ corresponds to one of the vertices of degree $2$ in $K_{2,5}$. By 
pigeonhole 
principle, one of $H_1$ and $H_2$ contains, for at least three 
distinct $i\in I$, neither $r_i$ nor $w_i$. Let us suppose without loss of 
generality that $H_1$ does so and let us denote three such $i\in I$ where 
$H_1$ contains neither $r_i$ nor $w_i$ as 
$i_1<i_2<i_3$. Now the set $\{r_{i_2},w_{i_2}\}$ disconnects every $r_j,w_j$ 
with $j<i_2$ from every $r_k,w_k$ with $k>i_2$. Therefore, as both, one of 
$r_{i_1},w_{i_1}$ and one of $r_{i_3},w_{i_3}$ correspond to one of the 
vertices of degree $2$ in the $K_{2,5}$, both $\{r_{i_1},w_{i_1}\}$ and 
$\{r_{i_3},w_{i_3}\}$ are adjacent to $H_1$. But, since $H_1$ 
is disjoint from $\{r_{i_2},w_{i_2}\}$, this contradicts the fact that $H_1$ is 
connected, as $H_1$ would need to 
meet two components of $G-r_{i_2}-w_{i_2}$. 
\end{proof}

% \begin{THM}
%   For every $k\in\N$ there is some $f(k)$ such that every graph with an 
% agile set of size $f(k)$ has a bond of size $\geq k$.
% \end{THM}
% 
% \begin{proof}
%   Let $G$ be a graph with an agile set $A$ and $|A| \geq f(k)$ and let $T$ 
% be a normal spanning tree of $G$.
%   Without loss of generality $G$ is 2-connected.
% 
%   If for some $v\in G$ at least $2k$ components of $T-v$ meet $A$ such a 
% bond is easy to find.
% 
%   If no such vertex exists, then there is a long normal path $P$ in $T$, 
% such that either $P$ meets $A$ often or many vertices on $P$ are adjacent in 
%$T$ 
% to components of $T-P$ meeting $A$.
% 	In either case partition $A$ in an alternating fashion along $P$. 
% \textellipsis
% \end{proof}

\section{Proof of \texorpdfstring{\cref{thm:agile_character}}{the main result}}\label{sec:characterizing}
We will be able to prove that \cref{cex:noK25} is effectively the only counterexample to \cref{question:agile_k2n}. More precisely, in this section we prove the following:
\agilecharacter*

To prove this, we will rely heavily on a result by Ding which characterizes 
the graphs not containing $K_{2,k}$ as a minor \cite{DingK2n}. Consequently, we 
need the following definitions by Ding \cite{DingK2n}:

We say that a graph $G$ is \emph{internally 3-connected} if we can obtain $G$ 
from a 3-connected graph by subdividing each edge at most once.

A \emph{fan} is a graph $G$ which consists of a cycle $C$, three consecutive 
vertices $a,b,c\in C$ and additional edges between $b$ and some other 
vertices 
on $C$. These additional edges are called the \emph{chords} of the fan. The 
vertex $b$ is called the \emph{center} of the fan, and the vertices $a$, $b$, $c$ are 
the \emph{corners} of the fan. The \emph{length} of the fan is the number of 
chords.

Consider graphs $G$ obtained from a cycle $C$ containing two disjoint edges 
$ab$ 
and $cd$ by adding some edges between the two distinct paths in $C\sm 
\{ab,cd\}$. 
The 
added edges are called \emph{chords}. We say that chords $f_1f_2$ and $f_3f_4$ 
\emph{cross} if the four vertices $f_1,f_2,f_3,f_4$ are pairwise distinct, and 
they appear in the order $f_1,f_3,f_2,f_4,f_1$ along $C$. 
If given such a graph $G$ where every chord is crossed by at most one other 
chord, and where if two chords $f_1f_2$ and $f_3f_4$ cross, then either $f_1f_3$ and 
$f_2f_4$, or $f_1f_4$ and $f_2f_3$ are edges in $C$, we call $G$ a 
\emph{strip}. 
Moreover, we will also call any $H\in \{G-ab,G-cd,G-ab-cd\}$ a \emph{strip} if 
$H$ has minimum degree at least $2$. The \emph{corners} of such a strip are the 
vertices $a,b,c,d$ and the \emph{length} of the strip is the maximal size of a 
set of pairwise non-crossing chords with pairwise disjoint endpoints.

Given a graph $G$, \emph{adding} a fan or strip to $G$ shall mean that we
obtain a new graph out of the disjoint union of $G$ and a fan or strip by 
identifying the corners of the fan or strip with disjoint vertices from $G$. 

We say that a graph $H$ is an \emph{augmentation} of a graph $G$ if 
$H$ 
is obtained from $G$ by adding disjoint fans and strips in such a way that two 
corners of distinct fans and strips are only allowed to be identified with the 
same vertex of $G$ if one of them is the center of a fan, and the other one is 
either a corner of a strip, or also a center of a fan.

We denote, for $m\in \N$, as $\cA_m$ the class of all graphs that are 
augmentations of a graph with at most $m$ vertices, i.e.\ the class of all those
graphs $H$ for which there is a graph $G$ with at most $m$ vertices such that 
$H$ is an augmentation of $G$.

A \emph{regular strip} of length $k$ is the graph obtained from two 
disjoint paths $P_1\coloneqq v_1\dots v_k$, $P_2\coloneqq  w_1\dots w_k$ by adding an edge 
between $v_i$ and $w_{i+1}$ and $w_i$ and $v_{i+1}$ for every $1\le i<m$. This 
graph is depicted in the following image:
 
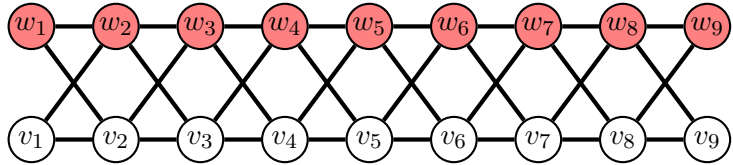
\begin{center}
\begin{tikzpicture}
\tikzstyle{hvertex}=[thick,circle,inner sep=0.cm, minimum size=6mm, 
fill=white, draw=black]
\tikzstyle{hedge}=[ultra thick]
	\def\npoints{9}
	\pgfmathtruncatemacro{\nminusone}{\npoints - 1}
	\pgfmathtruncatemacro{\nminustwo}{\npoints - 2}
	\foreach \i in {1,..., \npoints}{
		\node[hvertex,fill=red!50!white] (w\i) at (\i * 
10cm/\npoints,4cm){$w_{\i}$};
	}
	\foreach \i in {1,..., \npoints}{
		\node[hvertex] (v\i) at (\i * 10cm/\npoints,2.5cm){$v_{\i}$};
	}
	\foreach \i in {1,..., \nminusone}{
		\pgfmathtruncatemacro{\iplus}{\i + 1}
		\draw[hedge] (v\i) -- (v\iplus);
		\draw[hedge] (v\i) -- (w\iplus);
		\draw[hedge] (w\i) -- (v\iplus);
		\draw[hedge] (w\i) -- (w\iplus);
	}
\end{tikzpicture}
\end{center}
Note that such a regular strip is a strip with corners $v_1,w_1,v_k,w_k$.

The result of Ding \cite{DingK2n} now states the following:
\begin{THM}[{\cite{DingK2n}*{Theorem 5.1}, rephrased}]\label{thm:Ding}
  For every $k \in \N$ there is some $m \in \N$ such that
  every internally 3-connected graph without a $K_{2,k}$-minor is contained 
in $\cA_m$.
\end{THM}

This theorem will allow us to prove \cref{thm:agile_character}. Our proof 
strategy will be as follows: suppose we 
are given a graph $G$ which contains a large agile set. We will be able to show 
that, if $G$ does not contain a $K_{2,k}$-minor, then we find a minor of $G$ 
which is internally $3$-connected and still contains a relatively large agile 
set, hence we can essentially assume without loss of generality that $G$ is 
internally $3$-connected. We then can assume, using 
\cref{thm:Ding}, that $G$ lies in $\cA_m$, thus $G$ is an augmentation of a 
graph with at most $m$ vertices. Since every such 
augmentation is obtained by adding a bounded number of fans and strips, this 
will then imply that one of the fans or strips used for $G$ still contains a 
relatively large agile set, and we will be able to show that this is only 
possible for a strip, and that this strip will then contain a regular strip as 
a 
minor.

So let us first show that the graphs in $\cA_m$ are indeed constructed by 
adding 
only a bounded number of fans and strips:
\begin{OBS}\label{fact:bounded_strips}
 Every graph in ${\cal A}_m$ is obtained from a graph $G$ with at most $m$ 
vertices by adding at most $\frac{m}{2}$ fans and strips.
\end{OBS}
\begin{proof}
 Every vertex of $G$ is a non-center corner of a strip or fan for at most one 
such fan or strip. Since every strip and fan has at least $2$ non-center 
corners, this gives the desired bound.
\end{proof}

Our next goal is to show that, if an augmentation contains a large agile set, 
this large agile set cannot be contained in any of the fans used in the 
construction of 
this augmentation. Since the corners of such a fan separate the fan 
from the rest of the graph, this follows from \cref{fact:agile_seps} -- as soon 
as we establish the following:
\begin{OBS}\label{fact:agile_fan}
Let $G$ be a graph obtained from a fan by making the set of corners complete. 
Then $G$ cannot contain an agile set of size $7$.
\end{OBS}
\begin{proof}
Suppose $X$ is an agile set in $G$ of size at least $7$ and let us denote the 
center of that fan as $b$. Then there are $4$ vertices in $X$ which do not 
belong to the corners of the fan. Let us denote them as $v_1,v_2,v_3,v_4$ and 
assume that they lie in this order on the cycle $C$ used in the construction of 
$G$. 

Since $X$ is agile, there are disjoint paths $P_1$ from $v_1$ to $v_3$ 
and $P_2$ from $v_2$ to $v_4$. However, as $\{b,v_2,v_4\}$ together separate 
$v_1$ from $v_3$ in $G$ and $P_1$ can neither contain $v_2$ nor $v_4$, it 
needs to be the case that $b$ is contained in $P_1$. But similarly, 
$\{b,v_1,v_3\}$ separates $v_2$ from $v_4$ in $G$ and thus $b$ is contained in 
$P_2$, contradicting the fact that $P_1$ and $P_2$ are disjoint.
\end{proof}
\begin{COR}\label{lem:containing_fan}
 Let $G$ be a graph containing a fan $H$ as a subgraph such that the corners of 
$H$ separate the rest of $H$ from the rest of $G$. If $G$ contains an agile 
set $X$, then $X$ cannot contain more than $6$ vertices from $H$.
\end{COR}
\begin{proof}
This is immediate from \cref{fact:agile_seps} and \cref{fact:agile_fan}.
\end{proof}
Next, we would like to show, given an augmentation $G$ and a strip used in 
its augmentation process, that, if the vertices of that strip in $G$ contain a 
large agile set, then this strip needs to contain a large 
regular strip as a minor. 

For this, we first observe that a strip contains a regular strip as a minor as 
soon as the strip has enough pairs of crossing chords:
\begin{LEM}\label{lem:crossing_chords_regular}
 Let $G$ be a strip containing $k$ distinct pairs of crossing chords. Then $G$ 
contains a regular strip of length $k$ as a minor.
\end{LEM}
\begin{proof}
Let $C$ be the cycle used in the construction of $G$, and let $ab$ 
and $cd$ be the two edges of $C$ for which we added chords between $C\sm 
\{ab,cd\}$. Let us denote as $P_1$ and $P_2$ the two paths together forming
$C\sm\{ab,cd\}$, where $P_1$ starts in $a$ and $P_2$ starts in $b$.

Let us denote the pairs of crossing chords as 
$\{v_{1}^iw_{1}^i,v_{2}^iw_{2}^i\}$, where $v_j^i$ is contained in $P_1$, and where we 
enumerate these pairs so that $v_1^i$ appears before both, 
$v_{2}^i$ and $v_1^{i-1}$ on $P_1$, for every $1\le i<k$. Then, since every 
chord in a strip crosses at most one other chord, we have that $v_2^i$ 
appears before $v_1^{i+1}$ on $P_1$ (or is identical to $v_1^{i+1}$) and that 
the relation between the $w_j^i$ 
on $P_2$ is such that $w_2^i$ appears before $w_1^i$ which appears before 
$w_2^{i+1}$ on $P_2$ (again, or that $w_1^i=w_2^{i+1}$), for every $1\le i<k$. 

Thus, suppressing every vertex on $P_1$ or $P_2$ which is not one of the 
$v_j^i$ or $w_j^i$ and then contracting every existing edge between $v_2^i$ and 
$v_1^{i+1}$ as well as every edge between $w_1^i$ and $w_2^{i+1}$ for every 
$1\le i<k$ gives the desired regular strip of length $k$.
\end{proof}
With this lemma at hand we can now show the following:
\begin{LEM}\label{lem:agile_strips}
 There is a function $j\colon\N\to\N$, namely $j(k)\coloneqq 22(k-1)+4k+1$, such that 
whenever the graph $G$ obtained from 
a strip $H$ by making the set of corners of $H$ complete contains an agile set 
of size $j(k)$, then $H$ contains a regular strip of length $k$ as a minor.
\end{LEM}
\begin{proof}
 By \cref{lem:crossing_chords_regular}, if $H$ contains at least $k$ pairs of 
crossing chords, then $H$ contains a regular strip of length $k$ as a minor. So 
suppose that $H$ contains at most $k-1$ pairs of crossing chords. Let 
$X$ be the set of vertices incident with the edges of these chords together 
with the four corners of the strip. It is easy to see 
that $G-X$ contains at most $2(k-1)$ components, and that each of these 
components is either a path or a strip without crossing chords.
Moreover, each of these components is adjacent to at most $4$ vertices in $X$.
By the pigeonhole principle, one of these 
components needs to contain, since $j(k)>22(k-1)+4k$, at 
least $11$ vertices of our agile set. Let $Y$ be the vertex set of one 
such component, let $K'$ be the subgraph of $G$ induced on $Y\cup N(Y)$ and 
denote as $K$ the graph obtained from $K'$ by adding all edges between the 
vertices in $N(Y)$. Since $N(Y)$ separates $Y$ from the rest of $G$, it is, by 
\cref{fact:agile_seps}, enough to show that $K$ does not contain an agile set 
of 
size at least $11$.

The \emph{ladder of length $n$} is the graph on the set $[n]\times \{0,1\}$ 
where we add an edge between $(x,y)$ and $(x',y')$ precisely if 
$\abs{x-x'}+\abs{y-y'}=1$. The \emph{end vertices} of a ladder of length $n$ 
are the four vertices $(0,0),(0,1),(n-1,0)$ and $(n-1,1)$.

We claim that $K$ is a minor of a graph obtained from a ladder of large enough 
length by turning the set of the four end vertices of the ladder into a clique,
as depicted in \cref{fig:cycle_ladder}.
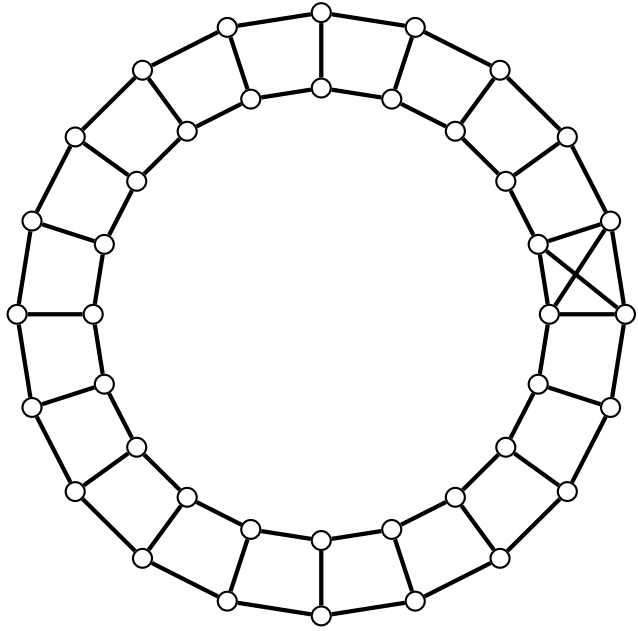
\begin{figure}[h]
 \begin{center}
\begin{tikzpicture}
\tikzstyle{hvertex}=[thick,circle,inner sep=0.cm, minimum size=2.5mm, 
fill=white, draw=black]
\tikzstyle{hedge}=[ultra thick]
	\def\npoints{20}
	\pgfmathtruncatemacro{\nminusone}{\npoints - 1}
	\pgfmathtruncatemacro{\nminustwo}{\npoints - 2}
	\foreach \i in {1,..., \npoints}{
		\node[hvertex] (w\i) at (\i * 
360/\npoints:3cm){};
	}
	\foreach \i in {1,..., \npoints}{
		\node[hvertex] (v\i) at (\i * 360/\npoints:4cm){};
	}
	\foreach \i in {1,..., \nminusone}{
		\pgfmathtruncatemacro{\iplus}{\i + 1}
		\draw[hedge] (v\i) -- (v\iplus);
		\draw[hedge] (w\i) -- (w\iplus);
		\draw[hedge] (w\i) -- (v\i);
	}
	\draw[hedge] (v\npoints) -- (w\npoints);
	\draw[hedge] (v1) -- (w\npoints);
	\draw[hedge] (v1) -- (v\npoints);
	\draw[hedge] (w1) -- (w\npoints);
	\draw[hedge] (w1) -- (v\npoints);
\end{tikzpicture}
\end{center}
\caption{The type of graph of which $K$ is a minor.}\label{fig:cycle_ladder}
\end{figure}
Indeed, if $G[Y]$ is a strip without crossing chords, then the graph obtained 
from $K'$ by 
removing the edges between the vertices from $N(Y)$ is also such a strip. 
Moreover, the corners of this strip are the vertices from $N(Y)$. Now we find 
this strip as a minor in a large enough ladder, with the additional property 
that the branch sets of the vertices from $N(Y)$ each contain one of the 
end vertices of the ladder. Consequently, $K$ is a minor of the graph obtained 
from that ladder by adding all edges between the end vertices of that ladder.

If on the other hand $G(Y)$ is a path, then again we find the graph obtained 
from $K'$ by removing the edges between the vertices from $N(Y)$ as a minor in 
a large enough ladder, with the additional property that the branch sets of 
$N(Y)$ each contain an end vertex from that ladder.

But, for any $n\in \N$, the graph obtained from a ladder of length $n$ 
by making its set of end vertices complete does not contain an agile set of size 
$11$: by pigeonhole principle, one of the rails of the ladder (that is one of the 
sets $\{(i,0)\mid i\in [n]\}, \{(i,1)\mid i\in [n]\}$) needs to contain 
$6$ vertices from our agile set. But if we partition these vertices 
alternatingly along the rail, we see that is not possible to connect the two partition classes 
disjointly.

Since containing an agile set is a minor-closed property, $K$ can thus also not 
contain an agile set of size $11$ contradicting, by 
\cref{fact:agile_seps}, the assumption that $Y$ contains $11$ vertices 
from our agile set.
\end{proof}

\begin{COR}\label{lem:force_regular_strip}
 Let $G$ be a graph containing a strip $H$ as a subgraph, such that the 
corners of $H$ separate the rest of $H$ from the rest of $G$. If $G$ 
contains an agile set $X$ containing more than $j(k)$ many vertices 
from $H$, then $G$ contains a regular strip of length $k$ as a minor.
\end{COR}
\begin{proof}
 This is immediate from \cref{fact:agile_seps} and \cref{lem:agile_strips}.
\end{proof}

We are now ready to use \cref{thm:Ding} to show that every internally 
$3$-connected graph containing a large agile set will indeed need to contain a 
large $K_{2,k}$ or a large regular strip as a minor:
\begin{LEM}\label{lem:internally_3}
  There exists a function $g\colon \N \to \N$ such that every internally 
3-connected graph which contains an agile set of size $g(k)$ contains $K_{2,k}$ 
or a regular strip of length $k$ as a minor.
\end{LEM}
\begin{proof}
By \cref{thm:Ding}, we find some $m\in \N$ such that every graph without a 
$K_{2,k}$-minor is contained in $\cA_m$. Let 
$g(k)\coloneqq m+\frac{m}{2}\max\{j(k),6\}$. Now let $G$ be a graph containing an agile 
set of size $g(k)$. By 
\cref{thm:Ding} we know that $G$ either contains a $K_{2,k}$ as a minor or 
is contained in $\cA_m$. By 
\cref{fact:bounded_strips}, in the second case $G$ is 
obtained from a graph of at most $m$ vertices by augmenting at most 
$\frac{m}{2}$ fans and strips. By the pigeonhole principle, one of the 
augmented fans or strips needs to contain an agile set of size at least 
$\max\{j(k),6\}$. 
By \cref{lem:containing_fan}, this cannot be a fan, so it needs to be a strip. 
However, by \cref{lem:force_regular_strip}, this implies that $G$ contains a 
regular strip of length $k$ as a minor.
\end{proof}

In order to extend this result to graphs that are not locally $3$-connected, it 
will be essential to analyse how separations of order $2$ in a graph containing 
a large 
agile set can behave. As it turns out, we are able to assume that they are all 
pairwise nested:
\begin{LEM}\label{lem:seps_nested}
 Let $n\in\N$ and let $G$ be a graph which is minor-minimal with the 
property of containing an 
agile set $X$ of size $n$,
 then the separations of order $2$ in $G$ form a nested set.
\end{LEM}
\begin{proof}
  Suppose that two separations $(A,B)$ and $(C,D)$ of order $2$ in $G$ cross.
  If $X$ meets all quadrants $ V \sm (B \cup D), V \sm (A \cup D), V \sm (B 
\cup C)$ and $V \sm (A \cup C)$,
  then this would contradict the agility of $X$ by partitioning vertices in 
opposite quadrants into the same partition class.

  Thus, at least one quadrant, say without loss of generality $V \sm (B \cup 
D)$, contains no vertex from $X$.
  Then, by minor-minimality of $G$, the quadrant contains no vertex, since 
contracting any edge adjacent to such a vertex results in a minor of $G$ in 
which $X$ is still 
agile. Thus, we may assume that $G[A \cap C]$ consists of just an edge between 
the sole vertices $v \in (A \cap B) \sm D$ and $w \in (C \cap D) \sm B$.

  We denote the second vertex in $A \cap B$ as $v'$ and the second vertex in 
$C \cap D$ as $w'$.

  We claim that we can contract the edge $vw$, contradicting the minimality 
of~$G$.

  Suppose first, that at most one of $v, w$ is contained in $X$, say $v \notin 
X$.
  Suppose that $X$ is not agile in $G'=G / vw$ and let us denote the 
partition of $X$ which witnesses this as $X = X_1 \cupdot X_2$. Since $X$ is 
agile 
in 
$G$, there are connected disjoint subgraphs $G_1, G_2$ of $G$ such that 
$X_1\subseteq G_1$ and $X_2\subseteq G_2$. Moreover, we may 
assume without loss of generality that $V(G_1)\cup V(G_2)=V$ and that $v \in 
G_1$ and 
$w \in G_2$, as otherwise $G'[V(G_1)]$ and $G'[V(G_2)]$ are also connected 
disjoint subgraphs of $G'$.

Now if $w' \in G_1$, then $B \cap C \supseteq G_2$. Then, for $V_1' = V(G_1) - 
v$ and $V_2' = V(G_2) + v$, we have that $G'[V_1']$ and $G'[V_2']$ are 
connected and 
contain $X_1$ and $X_2$, respectively. Thus, $X_1\cupdot X_2$ does not witness 
that 
$X$ is not agile in $G'$.

So suppose $w' \in G_2$. Then $v' \in G_2$, since $G_2$ is connected and 
$\{v,v'\}$ separates $w$ from $w'$.
By a symmetric argument to the above we may assume that $w \in X$, as otherwise 
$G'[V(G_1)+w]$ and $G'[V(G_2)-w]$ are connected and contain $X_1$ and $X_2$, 
respectively.
Thus, we have that $X_1 \subseteq V \sm (D \cup A)$.

If $X\cap B\cap C\subseteq X_1$, the partition $X=X_1\cupdot X_2$ would again 
not 
witness that $X$ is not agile in $G'$, as both $G'[B\cap C]$ and $G'[V\sm 
(B\cap C)]$ are connected. So we may suppose that $X_2\cap (B\cap C)$ is 
non-empty.

Moreover, there do not exist connected subgraphs $G_1''$ and $G_2''$ of 
$G[B\cap C]$ such that $v\notin G_1''$ and such that $X_1\subseteq G_1''$ and 
$X_2\cap (B\cap C)\subseteq G_2''$, as we could otherwise replace $G_1\cap
B\cap C$ and $G_2\cap B\cap C$ with these subgraphs which then shows that 
$X_1\cupdot X_2$ is not a partition witnessing that $X$ is not agile in $G'$.

 Additionally, $X$ is not completely contained in $C$, as $G$ was chosen 
$\subseteq$-minimal and $X\cap C$ is agile in the torso obtained from $G[C]$ 
by  \cref{fact:agile_seps}, and this torso is a proper minor of $G$ since there 
exists the vertex $v'\in (A\cap B)\sm C$. 

Thus, let $x \in X \sm C$ and let $X_1' = X_1 + x$ and $X_2' = X_2 - x$. Since 
$X$ is agile in $G$, there are disjoint connected subgraphs $T_1'$ and $T_2'$ 
of $G$ with $X_1'\subseteq T_1'$ and $X_2'\subseteq T_2'$.
Now $w' \in T_1'$, since $\{w, w'\}$ separates $X_1'$ in $G$ and $w \in 
X_2$.
  On the other hand $v \in T_1'$ since $X_1', X_2'$ look the same on $G[B \cap 
C]$ as $X_1, X_2$.
  But then $T_2$ cannot connect $w$ to $X_2 \cap (B \cap C)$, which is a 
contradiction.

  It thus remains the case that $v, w \in X$. Then $v', w' \notin X$, as 
$\{v,v'\}$ separates $w$ from $V\sm A$ and $\{w,w'\}$ separates $v$ from $V\sm 
C$ and both, $V\sm A$ and $V\sm C$ need to contain vertices from $X$ by the 
minor minimality of $G$.

We may assume that $X$ meets both $V \sm (A \cup D)$ 
and $V \sm (C \cup B)$ in vertices $x$ and $y$, say, as otherwise the 
corresponding quadrant would contain only an edge between $v$ and $w'$, 
respectively $w$ and $v'$, which could be contracted by the previous argument.

Now consider a partition $X_1'\cupdot X_2'$ of $X$ where $w, x \in X_1'$ and 
$v, y 
\in 
X_2'$. This partition witnesses that $X$ is not agile, which is a
contradiction.
\end{proof}

Thus, the set of all separations of order $2$ of such a minor-minimal $G$ which 
are neither small nor 
co-small form a tree-set, which in turn gives us a tree-decomposition 
$(T,\cV)$ of $G$ along all these separations. In particular all torsos of 
this tree-decomposition are $3$-connected. If a large subset of our agile set 
is contained in one of the 
torsos of this decomposition, we know by \cref{fact:agile_seps}, that it is 
still agile in the torso, and since the torso is $3$-connected, we can then 
apply \cref{lem:internally_3} to deduce that the torso, and thus also the 
original graph, contains a large $K_{2,k}$ or a large regular strip as a minor.

However, this does not need to be the case. But we can use the structure given 
by $T$ to analyse the case where it fails. Namely, if no torso contains a large 
agile set, then our agile set needs to be spread across a lot of different 
bags of the decomposition. This can either be the case in a `path-like' or 
in a `star-like' way. Let us first deal with the `path-like' case:
\begin{LEM}\label{lem:sep_sequence}
There exists a function $h\colon \N\to\N $ such that the following holds:
if $G$ is a graph containing an agile set $X$ and a sequence \[{(A_1, B_1) 
\leq \dots \leq (A_{h(k)}, B_{h(k)})}\] of separations of order $2$,
  so that $ (B_i \setminus B_{i+1}) \cap X \neq \emptyset $ for all $ 1 \leq 
i < h(k)$, then $G$ contains a $ K_{2,k} $ or a regular strip of length $k $ 
as 
a minor.
\end{LEM}

In order to \cref{lem:sep_sequence}, we will need to analyse how the parts of $G$ in-between these separations of order 2 look like.
We will want to find, as a substructure of these parts, a specific arrangement of paths which we call a `cross'.
Formally, a \emph{cross between} $A=\{a_1, a_2\}$ and $B=\{b_1, b_2\}$ consists of two disjoint $A$--$B$-paths $P_1$ and $P_2$ and two disjoint $P_1$--$P_2$-paths $Q_1$ and $Q_2$ so that the end vertex of $Q_1$ appears along $P_1$ before the end vertex of $Q_2$ and the end vertex of $Q_2$ appears along $P_2$ before the end vertex of $Q_1$.

The following two lemmas give two very specific sets of conditions under which we can find such a cross:

\begin{LEM}\label{lem:cross_free}
    Let $G$ be a graph and let $a_1, a_2, b_1, b_2$ be vertices of $G$.
    Let $P_1, P_2$ be disjoint paths, where $P_1$ is from $a_1$ to $b_1$ and $P_2$ is from $a_2$ to $b_2$,
    and let $Q_1, Q_2$ be disjoint paths, where $Q_1$ is from $a_1$ to $b_2$ and $Q_2$ is from $a_2$ to $b_1$.
    Moreover, let $G$ be edge-minimal with the property, that it contains such paths.
    Then $G$ contains a cross between $\{a_1, a_2\}$ and $\{b_1, b_2\}$.
\end{LEM}
\begin{proof}
    Let $v$ be the last vertex along $P_1$ for which $a_1P_1v$ is a subpath of $Q_1$, let $w$ the first vertex of $vQ_1 - v$ which meets $P_1 \cup P_2$, and let $Q_1'$ be the segment $vQ_1w$.
    By the edge-minimality of $G$, the vertex $w$ lies on $P_2$: if $w$ would lie on $P_1$, then we could delete the edges of the segment $vP_1w$ from $G$ and replace this segment of $P_1$ by $Q_1'$, contradicting the edge-minimality.
    Symmetrically, we let $Q_2'$ be the segment of $Q_2$ from the first point where it leaves $P_2$ to the first point where it meets $P_1 \cup P_2$ again and observe that it does so in $P_1$.

    The paths $P_1, P_2, Q_1'$, and  $Q_2'$ form a cross between $\{a_1, a_2\}$ and $\{b_1, b_2\}$.
\end{proof}

\begin{LEM}\label{lem:cros_wfree}
    Let $G$ be a graph and let $a_1, a_2, b_1, b_2, x$ be vertices of $G$.
    Let $S_1, S_2$ be disjoint trees, where $S_1$ is a path from $a_1$ to $b_1$ and $S_2$ has leaves precisely $a_2$, $b_2$, and $x$,
    and let $T_1, T_2$ be disjoint trees, where $T_2$ is a path from $a_2$ to $b_2$ and $T_1$ has leaves precisely $a_1$, $b_1$, and $x$.
    Moreover, let $G$ be edge-minimal with the property, that it contains such trees.
    Then $G$ contains a cross between $\{a_1, a_2\}$ and $\{b_1, b_2\}$.
\end{LEM}
\begin{proof}
    The trees $S_2$ and $T_1$ each have precisely one vertex of degree 3 which we call $s$ and $t$.
    If $a_1T_1t$ is not a subpath of $S_1$, then let $v$ be the last vertex along $a_1T_1t$ for which $a_1T_1v$ is a subpath of $S_1$ and let $w$ be the first vertex of $vS_1 - v$ which meets $T_1 \cap T_2$.
    We observe that, by the edge-minimality of $G$, the vertex $w$ lies on $T_2$.
    Let us denote the path $vT_2w$ as $P$.
    Consider the path $vS_2x$ and let $Q$ be the shortest subpath of $vS_2x$ which starts in $v$ and ends in $T_1$. (Such a path exists since $x\in T_1$.)

    By the edge-minimality of $G$, in $S_2$ there exists a $(a_2T_2w)$--$T_1$-path which we call $Q'$.
    If the end-vertex of $Q'$ is in $tT_1x$, then we extend $Q'$ in $T_1$ to $x$ to obtain $Q$,
    otherwise we let $Q = Q'$.

    Then $a_1T_1b_1$, $T_2$, $P$, and $Q$ form a cross between $\{a_1, a_2\}$ and $\{b_1, b_2\}$.

    If $a_1T_1t$ \emph{is} a subpath of $S_1$, then $tT_1b_1$ is not a subpath of $S_1$ and exchanging the sets $\{a_1, a_2\}$ and $\{b_1, b_2\}$ and proceeding as above gives the desired cross.
\end{proof}

\begin{proof}[Proof of \cref{lem:sep_sequence}]
 Let $x_i\in B_i\sm B_{i+1}\cap X$ and note that the $x_i$ are pairwise 
distinct. Let $A_i\cap B_i=\{s_i^1,s_i^2\}$.
 
 If we consider the partition of $X$ given by the two classes 
$X_1=\{x_1,x_3,\dots\}$ and $X_2=\{x_2,x_4,\dots\}$, and corresponding disjoint 
trees $T_1$ and $T_2$ containing $X_1$ and 
$X_2$, respectively, we observe that there need to be, for 
every $2<i<n-3$, two disjoint paths from $A_i\cap B_i$ to $A_{i+1}\cap B_{i+1}$ 
in $B_i\cap A_{i+1}$, one contained in $T_1$ and the other contained in $T_2$. 
 
 We say that the pair $\{i,(i+1)\}$ is \emph{free} if there are two pairs 
of such paths $T_1,T_2$, where one consists of a path between $s_i^1$ and 
$s_{i+1}^1$ and a path 
between $s_i^2$ and 
$s_{i+1}^2$, and the other consists of a path between $s_i^1$ and 
$s_{i+1}^2$ and a path between $s_i^2$ and $s_{i+1}^1$. Otherwise, the pair 
$i(i+1)$ is said to 
be \emph{restrictive}.
 
 We note that, if there are, for $2<i<n-2$, two consecutive pairs $\{(i-1),i\}$ 
and $\{i,(i+1)\}$ which both are restrictive, then for one of the pairs, say 
$\{i,(i+1)\}$, 
there need to be two pairs $T_{1,i},T_{2,i}$ and $T_{1,i}',T_{2,i}'$ of 
disjoint trees in 
$B_i\cap A_{i+1}$, such that $T_{1,i}$ contains $s_i^1,x_i,s_{i+1}^1$ and
$T_{2,i}$ 
contains $s_i^2,s_{i+1}^2$ and $T_{1,i}'$ contains $s_i^1,s_{i+1}^1$ and 
$T_{2,i}'$ 
contains $s_i^2,x_i,s_{i+1}^2$. This is due to the fact that $X$ is agile, and 
we can thus consider the partition of $X$ obtained from the one above by 
changing only the class to which $x_i$ belongs. 
Moreover, we may suppose without loss of generality that $T_{2,i}$ and 
$T_{1,i}'$ are paths.
 
 If there are, for the pair $\{i,(i+1)\}$, these four trees 
$T_{1,i},T_{2,i},T_{1,i}',T_{2,i}'$ as above with the additional property that 
the two paths 
$T_{2,i}$ and $T_{1,i}'$ meet, then we say that the restrictive pair 
$\{i,(i+1)\}$ is \emph{weakly free}.
 
We may suppose that $h(k)$ is chosen such 
that there either is a large interval $l<i<m$ with the property that every 
pair $\{i,(i+1)\}$ from that interval is restrictive and not weakly free, or 
that there is a large collection of pairs which are all free or weakly free.
 
 In the former case, by the definition of weakly free and the above 
observation about two adjacent restrictive pairs, we may suppose that $m$ 
(and thus $h(k)$) is chosen such that there are at least $k$ pairs
$\{i,(i+1)\}$ from the interval $l<i<m$, for which we find trees 
$T_{1,i},T_{2,i},T_{1,i}',T_{2,i}'$ as above with the 
additional property that $T_{2,i}$ and $T_{1,i}'$ are disjoint. In that case we 
find a $K_{2,k}$-minor in $G$ as follows: 
the branch sets for the two vertices of degree $k$ each consist of a path 
between $A_l\cap B_l$ and $A_m\cap B_m$ formed by concatenating the $T_{2,i}$'s 
and $T_{1,i}'$'s, respectively. Now we find a path between $x_i$ 
and $T_{2,i}$ and a path between $x_i$ and $T_{1,i}'$ both contained in 
$T_{1,i}\cup T_{2,i}'$ and thus both contained in $B_i\cap A_{i+1}$. The 
union of these two paths form, for each of the $k$ restrictive pairs 
considered, the branch set of a vertex of degree $2$ in our $K_{2,k}$-minor.
 
 So we may suppose that there is a collection of at least $k$ pairs 
$\{i,(i+1)\}$ 
which are all free or weakly free.

 In this case, \cref{lem:cros_wfree,lem:cross_free} ensure that whenever a pair $\{i,(i+1)\}$ is free or 
weakly free, there exists a cross between $\{s_i^1,s_i^2\}$ and $\{s_{i+1}^1,s_{i+1}^2\}$ in $B_i\cap A_{i+1}$ .
Combining all these crosses we obtain a regular strip of length $k$ as a minor in~$G$.
\end{proof}
We are now ready to prove \cref{thm:agile_character}:
\agilecharacter*
\begin{proof}[Proof of \cref{thm:agile_character}]
By \cref{2-connected} we may suppose that $G$ is $2$-connected. Moreover, we 
may 
suppose that $G$ is minor-minimal with the property that $G$ contains an agile 
set $X$ of size $f(k)$.

 By \cref{lem:seps_nested}, the regular separations of order $2$ in $G$ form a 
nested set, thus by \cite[Theorem 4.8]{confing}
there is a tree-decomposition of $G$ which induces all these separations. We consider the parts of this decomposition containing vertices of our 
agile set. If there is one node of $T$ whose corresponding part contains at 
least $g(k)$ many vertices of our agile set, then, since the torso 
of this part is $3$ connected, this torso and thus $G$ contains a $K_{2,k}$ or a 
regular strip of length $k$ as a minor by 
\cref{lem:internally_3}.
 
So each part of this decomposition contains less than $g(k)$ many vertices of 
our agile set. Thus, by the pigeonhole principle, we can choose $f(k)$ such 
that 
we either find a sequence of separations of $G$ as in \cref{lem:sep_sequence}, 
or that there is a node $t$ if $T$ such that there are $n\gg k$ many different 
components of $T-t$ which each contain a vertex $t'$ such that the part of the 
tree-decomposition corresponding to $t'$ contains a vertex of our agile set 
which is not contained in the part corresponding to $t$.

In the first case, we are immediately done by \cref{lem:sep_sequence}, so 
suppose that there indeed is a node $t$ of $T$ 
such that there are $n\gg k$ different components of $T-t$ which each contain a 
vertex $t'$ such that the part of the tree-decomposition corresponding to $t'$ 
contains a vertex of our agile set which is not contained in the part 
corresponding to $t$.
 
 Let us denote the separations corresponding to the incoming edges from these 
components to $t$ as $(A_1,B_1),\dots, (A_n,B_n)$ and note that they form a 
star.
 
 We now ask whether $A_i\sm B_i$ contains at least two vertices of $X$, or if 
this set contains only one such vertex. If at least $k$ of the sets $A_i\sm 
B_i$ contain at least two vertices from our agile set, we consider a partition 
$X_1\cupdot X_2$ of 
$X$ where, for each of these $A_i\sm B_i$, we add one vertex from $X\cap A_i\sm 
B_i$ to $X_1$ and all others to $X_2$. This results in two trees $T_1$ and 
$T_2$, where each of the separators 
$A_i\cap B_i$ needs to contain one vertex from $T_1$ and one vertex from $T_2$. 
Moreover, $T_1$ and $T_2$ still need to be connected after deleting all the 
sets 
$A_i\sm B_i$ which contain two vertices from $X$. These two connected sets 
form the two vertices of high degree of a $K_{2,k}$. The vertices of degree $2$ 
can then by obtained from the components of the sets $A_i\sm B_i$, since each 
such component needs, as $G$ is $2$-connected, to send an edge to both vertices 
in $A_i\cap B_i$ and thus sends an edge to both $T_1$ and $T_2$.
 
 So suppose that at least $g(k)$ of the sets $A_i\sm B_i$ only contain 
one vertex 
from $X$. Then we 
know, since $G$ was chosen minor-minimal, that $A_i\sm B_i$ consists for each 
such $i$ of just 
this one vertex from $X$ and that this vertex is, as $G$ is $2$-connected, 
adjacent to both vertices in $A_i\cap B_i$. 
 Note that no two separations $(A_i, B_i)$, $(A_j, B_j)$ corresponding to incoming edges of $t$ can have 
the same separator $A_i \cap B_i = A_j \cap B_j$, since then the supremum $(A_i \cup A_j, B_i \cap B_j)$ of these two separations would also 
lie in our nested set, and would imply that $t$ has only degree $3$ in $T$. 
 
 Let us show that, if we contract, for every incoming edge to $t$ which 
corresponds to a separation $(A,B)$, one component of the set $A\sm B$ 
down to a single vertex, and delete all other components of $A\sm B$, we are 
left with an internally $3$-connected graph.
 For this we only need to show that $G$ does not contain any edge $e$ in $C\cap 
D$, as the torso corresponding to $t$ is $3$-connected. 
 So suppose $G$ does contain an edge $e\in C\cap D$. 
 Then, since $G$ was chosen minor-minimal, deleting this edge results in $X$ 
no longer being agile, say because of the partition $X_1\cupdot X_2$, which was 
independent in $G$ as witnessed by $T_1$ and $T_2$, and suppose that $T_1$ 
contains $e$.
Then $T_2$ would need to be contained entirely in $A\sm B$, as otherwise $T_2$ 
is 
disjoint from $A\sm B$, and thus replacing $e$ with a path between the two 
vertices in $A\cap B$ contained in $A$ results in 
$(X_1,X_2)$ 
being independent in $G-e$. Thus, since $T_2$ is contained in $A\sm B$, we may 
assume that $T_2$ consists of just one vertex $x$ from $X$, and $A\sm B=\{x\}$. 
But now, since $(X_1,X_2)$ is not independent in $G-e$, this implies that $x$ 
is a separator in $G-e$. But the only neighbours of $x$ are the vertices in 
$A\cap B$, and thus, since $G$ contains more than $3$ vertices, one of the two 
vertices in $A\cap B$ would be a separator of $G$, contradicting 
\cref{2-connected}.
 
Thus, if we contract for every incoming edge to $t$ with separation $(A,B)$ one 
component of the set $A\sm B$ down to a single vertex, and delete all other 
components of $A\sm B$, we are left with an internally $3$-connected graph. 
This graph still needs to contain an agile set of size $g(k)$, since the set 
of all the vertices from $X$ which are the unique vertex from $X$ 
in one of the sets $A_i\sm B_i$ is agile in this restricted graph. Thus, by 
\cref{lem:internally_3}, we again find 
a $K_{2,k}$ or a regular strip of length $k$ as a minor.
\end{proof}

Now \cref{thm:agile_character} also gives a proof of \cref{prop:agileK24}:
\begin{proof}[Proof of \cref{prop:agileK24}]
If $G$ contains a large enough agile set, then $G$ contains, by 
\cref{thm:agile_character}, either a $K_{2,4}$ or a regular strip of length $4$ 
as a minor. Such a strip, however, also contains $K_{2,4}$ as a minor, which 
proves \cref{prop:agileK24}.
\end{proof}

\section{Generalizations}\label{sec:generalizations}

Let us now look at possible variations of the notion of an agile 
set. One such natural variation is the following: instead of just 
partitioning our set $X$ into two subsets, we might allow partitions into more 
partition classes and try to connect the vertices in each of these 
classes disjointly. More precisely, let us say that, given a graph $G=(V,E)$ 
and an integer $m$, a 
set $X \subseteq V$ is \emph{$m$-agile} in $G$ if for every partition $X = 
X_1\cupdot\dots\cupdot X_m$ (where we allow empty partition classes) there are 
vertex-disjoint connected subgraphs 
$T_1,\dots, T_m \subseteq G$ such that $X_i \subseteq T_i$. So, $X$ is 
$2$-agile if and only if $X$ is agile.
If a set $X$ is $m$-agile for every $m$, we say that $X$ is \emph{dexterous}. 
Note that this is equivalent to $X$ being 
$\left\lceil\frac{\abs{A}}{2}\right\rceil$-agile. 

Again, containing an $m$-agile or a dexterous set is closed under the minor 
relation in that, if $H$ is a minor of $G$ and $H$ contains an 
$m$-agile or dexterous set of size $k$, say, then $G$ also contains an 
$m$-agile 
or dexterous set of size $k$.

We again try to characterize, qualitatively, the existence of a large $m$-agile 
or dexterous set via a minor. A natural graph containing an $m$-agile set of 
size $k$ is the complete bipartite graph $K_{m,k}$. On the other hand the 
complete graph $K^m$ contains a dexterous set of size $m$. 

Another example of such graphs can be found in grids and, since neither $K^5$ 
not $K_{3,3}$ is a minor of the grid, these grids are another class of graphs 
containing a large dexterous or $m$-agile set. For dexterous sets we need to 
take a quadratic grid:
\begin{EX}
  In the $N^2\times N^2$--grid taking every $N$\textsuperscript{th} vertex of 
the diagonal gives a dexterous set of site $N$.
\end{EX}
Of course, since every dexterous set is $m$-agile for every $m$, the quadratic 
$N^2\times N^2$-grid also contains an $m$-agile set of size $N$. But such a set 
can actually already be found in a rectangular grid where the small side just 
needs to have size $2m-1$:
\begin{EX}
The $(2m-1)\times ((N-1)m+1)$ grid contains an $m$-agile set of size $N$:
if we denote the vertices of the grid by \[\{v_{i,j}\mid 1\le i \le 2m-1, 1\le 
j \le N(m-1)\},\] then the set $\{v_{m-1,jk+1}\mid 0\le j\le N-1\}$ is 
$m$-agile.
 How to construct the required trees is illustrated in the following 
picture: 
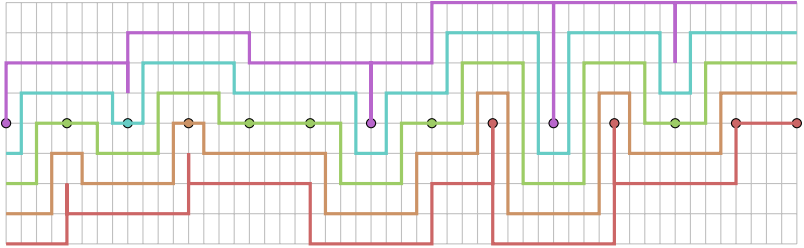
\begin{center}
\begin{tikzpicture}[rotate=90]
  \newcommand\gridwidth{4}
  \newcommand\agilesize{13}
  \edef\agileseed{464163546}
  \pgfmathsetmacro{\agilesizeminus}{\agilesize - 1}
  \pgfmathsetmacro\length{\gridwidth * \agilesize}
  \newcommand\xstep{.4cm}
  \newcommand\ystep{.2cm}
  \draw[ystep=\ystep,xstep=\xstep,lightgray] (-\gridwidth*\xstep,0) grid 
(\gridwidth*\xstep,\length*\ystep);
  \pgfmathsetseed{\agileseed}
  \pgfmathsetmacro{\class}{0}
  \foreach \y in {0,\gridwidth,...,\length} {
    \ifnum\y>0 \pgfmathrandominteger{\class}{0}{\gridwidth} \fi
    \pgfmathsetmacro{\hue}{(\class/\gridwidth)^1.715*0.8}
    \definecolor{pathcolor}{hsb}{\hue,.5,.8}
    \fill[pathcolor] (0, \y*\ystep) circle [radius=.3*\ystep];
    \draw[] (0, \y*\ystep) circle [radius=.3*\ystep];
  }
    \foreach \pi in {0,...,\gridwidth} {
      \pgfmathsetseed{\agileseed}
      \pgfmathsetmacro{\hue}{(\pi/\gridwidth)^1.715*0.8}
      \definecolor{pathcolor}{hsb}{\hue,.5,.8}
      \pgfmathsetmacro{\lastclass}{0}
      \draw[pathcolor,very thick] (\pi*\xstep,0)
      \foreach \pos [remember=\class as \lastclass] in {0, ..., 
\agilesizeminus} {
        \pgfextra
        \pgfmathrandominteger{\class}{0}{\gridwidth}
        \endpgfextra
        \ifnum \class>\lastclass
          |- (-\class*\xstep + \pi * \xstep,\pos*\gridwidth*\ystep + 
\pi * \ystep) -- (-\class*\xstep + \pi*\xstep, \pos*\gridwidth*\ystep + 
\gridwidth*\ystep)
        \else 
          |- (-\class*\xstep + \pi * \xstep,\pos*\gridwidth*\ystep + 
\gridwidth*\ystep - \pi * \ystep) -- (-\class*\xstep + \pi*\xstep, 
\pos*\gridwidth*\ystep + \gridwidth*\ystep)
        \fi
      };
    }

\end{tikzpicture}
\end{center}
\end{EX}
In particular, as every $m$-agile set of size $2m$ is dexterous, the 
${(N-1)\times (\frac{(N-1)N}{2}+1)}$-grid contains a dexterous set of size $N$.

If we seek only for a qualitative result, not just in terms of the size of our 
$m$-agile set but also in terms of $m$, we can 
actually show that, conversely, a large enough $l$-agile set, for large 
enough $l$, forces the existence of either a $K_{m,N}$ or a long rectangular 
grid as a minor. We will show this in \cref{thm:kagile_character}.

But first, we can similarly show that a large enough dexterous set 
forces the existence of a large complete graph or a large quadratic 
grid as a minor.
This result for dexterous sets can be obtained immediately, as the existence of a 
large dexterous set implies that our graph has high tree-width.
This can either be shown directly or by using a result by Diestel, 
Jensen, Gorbunov, and Thomassen \cite{excludedGrid} about so-called $m$-connected sets. 
Following their definition, a vertex set $X$ is \emph{$m$-connected} if 
$\abs{X}\ge m$ and for any two subsets $X_1,X_2\subseteq X$ with 
$\abs{X_1}=\abs{X_2}\le m$ one can find $\abs{X_1}$ many disjoint paths between 
$X_1$ and $X_2$.

This notion of $m$-connected vertex sets is related to our $m$-agile sets in 
that an $m$-agile set is $m$-connected:
\begin{PROP}\label{prop:agile_connected}
 If a graph contains an $m$-agile set $Z$ of size at least $m$, then $Z$ is 
$m$-connected.
\end{PROP}
\begin{proof}
 Let $X,Y\subseteq Z$ such that $\abs{X}=\abs{Y}\le m$. Then we can find 
pairs  of vertices $\{x_1,y_1\},\dots, \{x_{\abs{X}},y_{\abs{X}}\}$ so that $x_i\in X$ and 
$y_i\in Y$ and so that $x_i=y_j$ only if $i=j$. We now construct a partition 
of $Z$ into classes $X_1,\dots, X_{\abs{X}}$ by defining $X_i=\{x_i,y_i\}$ 
whenever $i\ge 2$ and $X_1=X\sm \{x_2,y_2,\dots,x_{\abs{X}},y_{\abs{X}}\}$.

Since $Z$ is $m$-agile, we thus find disjoint trees $T_1,\dots,T_{\abs{X}}$ 
so that $X_i\subseteq T_i$. In particular, $T_i$ contains a path between 
$x_i$ and $y_i$ which shows that $Z$ is $m$-connected.
\end{proof}
As Diestel, Jensen, Gorbunov, and Thomassen showed, the existence of a large 
$m$-connected vertex set is an obstruction to the graph having low tree-width, 
thus the same holds for dexterous sets as well. Concretely, they showed the 
following:
\begin{PROP}[\cite{excludedGrid}*{Proposition 
3(i)}]\label{prop:treewidth_connected_set}
 Let $G$ be a graph and $k>0$ an integer. If $G$ has tree-width $<k$ 
then $G$ contains no $(k+1)$-connected set of size~$\ge 3k$.
\end{PROP}
Together with the grid theorem by Robertson and Seymour \cite{GMV},
which states that a graph of large enough 
tree-width needs to contain an $N\times N$-grid as a minor, this directly 
implies a qualitative relation 
between the existence of a dexterous set and a grid minor:
\begin{THM}
There is a function $f\colon\N\to\N$ such that every graph containing a 
dexterous set of size at least $f(N)$ also contains the $N\times N$-grid as a 
minor
\end{THM}
\begin{proof}
If a graph contains a dexterous set of size at least $k$, it has, by 
\cref{prop:treewidth_connected_set} and \cref{prop:agile_connected}, tree-width 
at least $\frac{k}{3}-1$. However, by the grid minor theorem \cite{GMV} (see 
also \cite{excludedGrid}*{Theorem 2}) there is a 
function $f$ such that every graph of tree-width $>f(N)$ contains an $N\times 
N$-grid as a minor.
\end{proof}

For large $m$-agile sets we will be able to show that their existence is, 
again qualitatively, characterized by $K_{m,k}$ and rectangular grid minors. 
Again we can build on existing literature on variations of the grid theorem. 
This time we will utilize a generalized version of the grid 
theorem obtained by Geelen and Joeris \cite{GeneralizedGrid}*{Theorem 
9.3}. 
To state 
this theorem, we need the following additional definition from that paper. 
Given parameters $t,l,n$ a \emph{$(t,l,n)$-wheel} is a graph obtained from a 
tree $T$ 
with $t$ vertices, a set $Z$ of size $l$, a permutation $\pi\colon V(T)\to V(T)$ and 
a 
function $\psi\colon Z\to V(T)$ via the following construction:
we start with $n$ disjoint copies of $T$, called $T_1,\dots,T_n$. Let us denote 
the copy of $v\in V$ in $T_i$ as $v_i$. We then add an edge between $v_i$ and 
$v_{i+1}$ for any vertex $v\in V$ and any index $i$ between $1$ and $n-1$.
Then we add an edge between $v_n$ and $w_1$ where $w=\pi(v)$. As a last 
step, for every $z\in Z$ and $z=\psi(z)$, we add an edge between $z$ and every 
$v_i$.

A \emph{$(\theta,n)$-wheel} is any graph which is a $(t,l,n)$-wheel for some 
$t,l\in \N$ satisfying $2t+l=\theta$.

\cite{GeneralizedGrid}*{Theorem 9.3} by Geelen and Joeris now implies the 
following:
\begin{THM}[\cite{GeneralizedGrid}]\label{thm:generilsed_grid}
 There exists a function ${f\colon \N^2\to \N}$ such that, given 
$\theta,n\in \N$ 
with 
$\theta\ge 2$ and $n\ge 3$, every graph $G$ containing a $\theta$-connected set 
$U$ of size at least $f(\theta,n)$ contains a $K_{\theta,n}$ or a 
$(\theta,n)$-wheel as a minor.
\end{THM}
Using this result, we can now show that the existence 
of an $m$-agile set is indeed characterized by the existence of a large 
rectangular grid or a large complete bipartite graph as a minor. Concretely, 
we can show the following:
\begin{THM}\label{thm:kagile_character}
There is a function $f\colon \N^2\to\N$ such that every graph containing an 
$((m-1)2m+1)$-agile set of size at least $f(m,k)$ contains $K_{m,k}$ or the 
$((2m-1)\times k)$-grid as a minor.
\end{THM}
\begin{proof}
Every such $((m-1)2m+1)$-agile set is, by \cref{prop:agile_connected}, also
$((m-1)2m+1)$-connected, 
thus by \cref{thm:generilsed_grid} there is a function $f$ such that 
every graph containing an $((m-1)2m+1)$-agile set of size at least $f(m,k)$ 
either 
contains $K_{(m-1)2m+1,k}$ or an $((m-1)2m+1,k)$-wheel as a minor. We are now 
going to show that such a wheel contains a $K_{m,k}$ or a $(2m-1)\times k$-grid 
as 
a minor. For this recall that such a wheel was constructed using a tree $T$ of 
size $t$, say, and a set $Z$ of central vertices of size $z$, say, so that 
$2t+z=(m-1)2m+1$. 

Consider the tree $T'$ obtained from $T$ by adding every 
vertex in $Z$ as a leaf to $T$ in such a way that $z\in Z$ is adjacent to its 
neighbour $\psi(z)$ in $T$. Then $T'$ has at least 
$(m-1)m+1$ many vertices. Thus, $T'$ has at least $m$ leaves or 
contains a path of length at least $2m+1$. 

If $T'$ contains a set $L$ of $m$ 
leaves, we can construct a $K_{m,k}$-minor in $G$ as follows: 
for every leaf $v\in L$ of $T'$, if $v$ is a vertex of $T$ we let $X_v$ be the 
set of all $v_i\in T'$. If $v\in Z$ then $X_v=\{z\}$. Clearly every $X_v$ is 
connected and the sets $X_v$ will be the branch sets of the vertices of degree 
$k$ in our $K_{m,k}$-minor. For the vertices of degree $m$ we now take, for 
every $1\le i\le k$, the rest of $T_i$, i.e.\ the set $X_i=T_i\sm \bigcup_{v\in 
L} X_v$. Since every $v\in L$ is a leaf of $T$, the set $X_i$ is connected in 
the 
wheel. Moreover, each $X_i$ has a neighbour in $X_v$ for every $v\in L$ and the 
$X_i$ and $X_v$ are all pairwise disjoint, which completes the construction of 
our $K_{m,k}$-minor.

If on the other hand $T'$ contains a path of 
length at least $2m+1$, then $T$ needs to contain a path $P$ of length $2m-1$, 
as the vertices in $Z$ were only added as leafs to $T$. This $P$ directly 
corresponds to the $(2m-1)$-columns of a $(2m-1)\times k$-grid minor in $G$, 
i.e.\ 
the 
restriction of the wheel to the set of all those $v_i$ for which $v\in P$, 
equals an $(2m-1)\times k$-grid, except for some additional edges.
\end{proof}

We remark that it can actually be shown, with some amount of bookkeeping, that a 
$(t,l,k(\frac{t+l}{2}-1))$-wheel itself induces a 
$\left(\frac{t}{2}+l\right)$-agile set of 
size $k$.
Thus, we could as well have formulated \cref{thm:kagile_character} in terms of 
a corresponding wheel instead of a regular grid.

Let us end this paper with one final observation regarding 
\cref{question:agile_k2n}. While we have seen that the existence of a large 
$2$-agile set alone is not enough to guarantee the existence of a 
$K_{2,k}$-minor, due to the regular strips, it turns out that, by requiring the 
existence of a large $3$-agile set, we can actually guarantee the existence of 
a $K_{2,k}$-minor.

\begin{THM}
There exists a function $f\colon\N\to\N$ such that every graph containing a 
$3$-agile set of size at least $f(k)$ also contains a $K_{2,k}$-minor.
\end{THM}
\begin{proof}
Let $G$ be a graph containing a $3$-agile set $X$ of size $f(k)\coloneqq N>4$. Like in 
\cref{2-connected} we may assume that $G$ is $2$-connected: if $(A,B)$ is a 
separation of $G$ such that $A\cap B$ contains at most one vertex, then either 
$A$ or $B$ contains only one vertex of $X$, otherwise $X$ would not be 
$3$-agile. So suppose that $A$ contains at most one vertex from $X$, we claim 
that $G'\coloneqq G-(A\setminus B)$ is a subgraph of $G$ which also contains a 
$3$-agile set of size $N$. Indeed, it is easy to check that $(X\cap B) \cup (A\cap 
B)$ is such a $3$-agile set in $G'$.

Thus, by taking a minimal subgraph of $G$ containing a $3$-agile set of 
size $N$, we may suppose that $G$ is $2$-connected.

Let $T$ be a normal spanning-tree of $G$ with root $r$. Since $\abs{G}\ge N$, by taking $N > 9k^3$ we can ensure that $T$ either contains a vertex of degree at least 
$k+1$ 
or a path $P$ and at least $n = 9k^2$ vertices $v_1,\dots,v_n$ (enumerated 
starting from the root of $T$) on $P$ such that there is a component $C_i$ of 
$T-P$ with neighbour $v_i$ in $T$ (possibly the empty component) with the 
property that 
$C_i\cup \{v_i\}$ contains a vertex in $X$.

If $T$ contains a vertex of degree at least 
$k+1$, 
it is easy to find the desired $K_{2,k}$-minor, as follows: if $t\in T$ has 
degree at 
least $k+1$, there are $k$ distinct components of $G-t$ which do not contain 
the root. Now since $T$ is normal and $G$ is $2$-connected, each of these 
components needs to send an additional edge to the path $P$ in $T$ between $r$ 
and $t$. Thus, taking all the components as the branch sets of the degree-$2$ 
vertices, the vertex $v$ itself as the branch set of one of the degree-$k$ 
vertices and the vertices of $P-t$ as the branch set of the other degree-$k$ 
vertex gives the desired $K_{2,k}$-minor.

So we may suppose that $T$ does not contain a vertex of degree $k+1$, and thus 
contains a path $P$ starting in $r$ with $n$ vertices $v_1,\dots,v_n$ on $P$ such that there is a component $C_i$ (possibly the empty set) of 
$T-P$ with neighbour $v_i$ in $T$ with the 
property that 
$C_i\cup \{v_i\}$ contains a vertex in $X$, which we call~$x_i$.

We partition $X$ into three disjoint sets $A\cupdot B\cupdot C$ so that 
$A$ contains all $x_i$ where $i$ is divisible by $3$, and 
$B$ contains all $x_i$ where $i \equiv 1 \mod 3$,
and $C$ contains all $x_i$ where $i \equiv 2 \mod 3$. Since $X$ is $3$-agile, 
there 
are disjoint trees $T_0,T_1,T_2\subseteq G$ such that 
$A\subseteq T_0$, $B\subseteq T_1$, and $C\subseteq T_2$. We may assume without 
loss of generality that every vertex on $P$ belongs to either $T_0$ or $T_1$ or 
$T_2$. We say that a subpath $P'$ of $P$ is \emph{$T_l$-free} if no vertex on 
$P'$ is contained in $T_l$.

Now if $P$ contains a $T_l$-free subpath $P'$, for some $l$, which 
contains $3k$ of the $v_i$, say $v_j,\dots,v_{j+3k-1}$, then 
we find the desired $K_{2,k}$-minor: there are at least $k$ of the vertices 
$v_j,\dots,v_{j+3k-1}$, for which the corresponding $x_i$ lies in $T_l$. In 
particular, the corresponding $C_i$ are met by $T_l$. Since $T$ is a 
normal spanning tree, there therefore exists, for every such $C_i$, an adjacent 
vertex on $rTv_j - v_j$ which lies in $T_l$. But now we obtain our desired $K_{2,k}$-minor by 
taking as one of the vertices of high degree the subpath $P'$, as the other one 
the path $rTv_j  - v_j$ and as the vertices of degree $2$ the components $C_i$ 
mentioned above.
Hence, there cannot be such a $T_l$-free subpath containing $3k$ of the $v_i$, so every 
subpath $P'$ of $P$ which contains at least $3k$ of the $v_i$ meets all three $T_l$.

In particular, if we partition $P$ into subpaths $P_1, P_2,\dots, P_{3k}$ 
each containing $3k$ of the $v_i$, then each of the $P_j$ meets all the 
$T_l$. Thus, for each $P_j$ there is a $0\le l\le 2$ such that $P_j$ contains a 
subpath on which all $v_i$ are contained in $T_l$, put the preceding $v_i$ and 
the successive $v_i$ on $P$ together meet both the other trees.

By the pigeonhole principle we find that for at least $k$ of the $P_j$ the 
chosen $T_l$ is the same. But now we obtain a $K_{2,k}$-minor by taking the 
other two $T_l$s as the vertices of high degree, and use the paths $P_j$ found 
above as the vertices of degree 2 in our $K_{2,k}$-minor.
\end{proof}

\bibliography{agile}

\end{document}